\documentclass[10pt, a4paper]{amsart}
\usepackage{amssymb, amsmath, amsfonts, verbatim}
\usepackage{amsthm}
\usepackage[numbers]{natbib}
\usepackage{fullpage}
\usepackage{pdflscape}
\usepackage{color}
\newtheorem{lemma}[equation]{Lemma}
\newtheorem{theorem}[equation]{Theorem}

\newtheorem{definition}[equation]{Definition}

\theoremstyle{remark}
\newtheorem{remark}[equation]{Remark}

\newtheorem{conjecture}[equation]{Conjecture}
\newtheorem{question}[equation]{Question}

\newcommand{\BB}{\mathbb{B}}

\newcommand{\FF}{\mathbb{F}}
\newcommand{\MM}{\mathbb{M}}

\title{New examples of mixed Beauville groups}
\author{Ben Fairbairn, Emilio Pierro}
\date{}
\address{Department of Economics, Mathematics and Statistics, Birkbeck, University of London, Malet St, London WC1E 7HX\\ \texttt{b.fairbairn@bbk.ac.uk, e.pierro@mail.bbk.ac.uk}}

\begin{document}
\maketitle

\begin{abstract}
We generalise a construction of mixed Beauville groups first given by Bauer, Catanese and Grunewald. We go on to give several examples of infinite families of characteristically simple groups that satisfy the hypotheses of our theorem and thus provide a wealth of new examples of mixed Beauville groups.
\end{abstract}

\section{Introduction}

We recall the following definition first made by Bauer, Catanese and Grunewald in \cite[Definition 4.1]{BCG}.

\begin{definition} \label{mixed}
Let $G$ be a finite group and for $x,y\in G$ define
\[
\Sigma(x,y)=\bigcup_{i=1}^{|G|}\bigcup_{g\in G}\{(x^i)^g,(y^i)^g,((xy)^i)^g\}.
\]
A \textbf{mixed Beauville quadruple} for $G$ is a quadruple $(G^0;a,c;g)$ consisting of a subgroup $G^0$ of index $2$ in $G$; of elements $a,c\in G^0$ and of an element $g\in G$ such that
\begin{enumerate}
\item[M1] $G^0$ is generated by $a$ and $c$;
\item[M2] $g\not\in G^0$;
\item[M3] for every $\gamma\in G^0$ we have that $(g\gamma)^2\not\in\Sigma(a,c)$ and
\item[M4] $\Sigma(a,c)\cap\Sigma(a^g,c^g)=\{e\}$.
\end{enumerate}
If the $g$ is clear then we will omit this simply writing $(G^0;a,c)$ instead. If $G$ has a mixed Beauville quadruple we say that $G$ is a \textbf{mixed Beauville group} and call $(G^0;a,c)$ a \textbf{mixed Beauville structure} on $G$. Finally, the \textbf{type} of this structure is $(o(a_1),o(c_1),o(a_1c_1);o(a_2),o(c_2),o(a_2c_2))$.
\end{definition}

We will not discuss the corresponding `unmixed' Beauville groups here. Beauville groups were originally introduced in connection with a class of complex surfaces of general type, known as Beauville surfaces. These surfaces possess many useful geometric properties: their automorphism groups \cite{Jonesauts} and fundamental groups \cite{Catanese} are relatively easy to compute and these surfaces are rigid in the sense of admitting no non-trivial deformations \cite{BCG2} and thus correspond to isolated points in the moduli space of surfaces of general type. Early motivation came from providing cheap counterexamples to the so-called `Friedman-Morgan speculation' \cite{FM} but they also provide a wide class of surfaces that are unusually easy to deal with to test conjectures and provide counterexamples. A number of excellent surveys on these and closely related matters have appeared in recent years - see any of \cite{7,11,12,F,gaj,CombSurv} and the references therein.

To construct examples of mixed Beauville groups, Bauer, Catanese and Grunewald proved the following in \cite[Lemma 4.5]{BCG}.

\begin{theorem}\label{original}
Let $H$ be a perfect finite group. Let $a_1,c_1,a_2,c_2\in H$ and define $\nu(a_i,c_i)=o(a_i)o(c_i)o(a_ic_i)$ for $i=1,2$. Assume that
\begin{enumerate}
\item $o(a_1)$ and $o(c_1)$ are even;
\item $\langle a_1,c_1\rangle=H$;
\item $\langle a_2,c_2\rangle=H$;
\item $\nu(a_1,c_1)$ is coprime to $\nu(a_2,c_2)$.
\end{enumerate}
Set $G:=(H\times H):\langle g\rangle$ where $g$ is an element of order $4$ that acts by interchanging the two factors; $G^0=H\times H\times\langle g^2\rangle$; $a:=(a_1,a_2,g^2)$ and $c:=(c_1,c_2,g^2)$. Then $(G^0;a,c)$ is a mixed Beauville structure on $G$.
\end{theorem}

The only examples of groups satisfying the hypotheses of Theorem \ref{original} given by Bauer, Catanese and Grunewald were the alternating groups $A_n$ `for large $n$' (proved using the heavy duty machinery first employed to verify Higman's conjecture concerning alternating groups as images of triangle groups) and the groups $SL_2(p)$ with $p\not=2,3,5,17$ prime (though their argument also does not apply to the case $p=7$).

More generally, mixed Beauville groups have proved extremely difficult to construct. Bauer, Catanese and Grunewald show \cite{pq0} that there are two groups of order $2^8$ which admit a mixed Beauville structure but no group of smaller order, however, the method they use is that of checking every group computationally. Indeed, any $p$-group admitting a mixed Beauville structure must be a 2-group. Barker, Bosten, Peyerimhoff and Vdovina construct five new examples of mixed Beauville 2-groups in \cite{bbpv1} and an infinite family in \cite{bbpv2} and the aforementioned examples account for all presently known mixed Beauville groups. Non-examples of mixed Beauville groups, however, are in abundance. The following result is due to Fuertes and Gonz\'{a}lez-Diez \cite[Lemma 5]{fgd}.
\begin{lemma} Let $(C \times C)/G$ be a Beauville surface of mixed type and $G^0$ the subgroup of $G$ consisting of the elements which preserve each of the factors, then the order of any element $f \in G \setminus G^0$ is divisible by $4$. \end{lemma}
Fuertes and Gonz\'{a}lez-Diez used the above to prove that no symmetric group is a mixed Beauville group. It is easy to see, however that this result actually rules out most almost simple groups \cite{exc} (though as the groups $P\Sigma L_2(p^2)$ with $p$ prime show not quite all of them). Bauer, Catanese and Grunewald also show in \cite[Theorem 4.3]{BCG} that $G^0$ must be non-abelian.

In light of the above we make the following definition.

\begin{definition} \label{mixable}
Let $G$ be a group satisfying the hypotheses of Theorem \ref{original}. Then we say that $G$ is a \textbf{mixable Beauville group}. More specifically, $G$ is a mixable Beauville group if $G$ is a perfect group and there exist $x_1,y_1,x_2,y_2\in G$ such that $o(x_1)$ and $o(y_1)$ are even; $\langle x_1,y_1\rangle=\langle x_2,y_2\rangle=G$ and $\nu(x_1,y_1)$ is coprime to $\nu(x_2,y_2)$. We call $(x_1,y_1,x_2,y_2)$ a \textbf{mixable Beauville structure} for $G$ of \textbf{type} $(o(x_1),o(y_1),o(x_1y_1);o(x_2),o(y_2),o(x_2y_2))$.
\end{definition}

Thus any mixable Beauville group automatically gives us a mixed Beauville group by Theorem \ref{original}. In Section 2 we will prove a generalisation of Theorem \ref{original} that shows a given mixable Beaville group actually provides us with a wealth of mixed Beauville groups.

We remark that finding $2$-generated non-mixable Beauville groups is not difficult. For example, no symmetric group is mixable (every generating set must contain elements of even order, so we cannot satisfy the conditions in Definition \ref{mixable}); $p$-groups are not mixable (to have an index $2$ subgroup we must have $p=2$ and then again the conditions in Definition \ref{mixable} cannot be satisfied) and even among the finite simple groups, $PSL_2(2^n)$ for $n\geq2$ are not mixable (the only elements of even order have order $2$ and thus condition (2) of Theorem \ref{original} cannot be satisfied). Despite this, we prove the following.

\begin{theorem}\label{main}
If $G$ belongs to any of the following families of simple groups
\begin{itemize}
\item the alternating groups $A_n$, $n\geq6$;
\item the linear groups $PSL_2(q)$ with $q\geq7$ odd;
\item the unitary groups $U_3(q)$ with $q\geq3$;
\item the Suzuki groups $^2B_2(2^{2n+1})$ with $n\geq1$;
\item the small Ree groups $^2G_2(3^{2n+1})$ with $n\geq1$;
\item the large Ree groups $^2F_4(2^{2n+1})$ with $n\geq1$;
\item the Steinberg triality group $^3D_4(q)$ with $q\geq2$, or;
\item the sporadic groups (including the Tits group $^2F_4(2)'$),
\end{itemize}
then $G$ is a mixable Beauville group. Furthermore, if $G$ is any of the above or $PSL_2(2^n)$ with $n\geq3$ then $G\times G$ is also a mixable Beauville group.
\end{theorem}

In light of the above Theorem we make the following conjecture.

\begin{conjecture}
If $G$ is a finite simple group not isomorphic to $PSL_2(2^n)$ for $n\geq2$ then $G$ is a mixable Beauville group.
\end{conjecture}

Groups of the form $G^n$ where $G$ is a finite simple group are called characteristically simple groups. The study of characteristically simple Beauville groups has recently been initiated by Jones in \cite{JonesChar1,JonesChar2}. It is easy to show that the characteristically simple group $PSL_2(7)\times PSL_2(7)\times PSL_2(7)$ is not a mixable Beauville group. We thus ask the following natural question.

\begin{question}
If $G$ is a simple group then for which $n$ is $G^n$ a mixable Beauville group?
\end{question}

Throughout we use the standard `Atlas' notation for finite groups and related concepts as described in \cite{ATLAS}. In Section 2 we discuss a generalization of Theorem \ref{original} which enables mixable Beauville groups to define several mixed Beauville groups before turning our attention in the remaining Sections to the proof of Theorem \ref{main}.

\section{Generalisation and auxiliary results}%%%%%%%%%%%%%%%%%%%%%%%%%%----GENERALISATIONS AND AUXILIARY RESULTS----%%%%%%%%%%%%%%%%%%%%%%%%%%
We begin with the following definition which will be key to the generating structures we demonstrate.

\begin{definition} Let $G$ be a finite group and $x, y, z \in G$. A \textbf{hyperbolic generating triple} for $G$ is a triple $(x,y,z) \in G \times G \times G$ such that \begin{enumerate}
\item $\frac{1}{o(x)} + \frac{1}{o(y)} + \frac{1}{o(z)} < 1$,
\item $\langle x, y, z \rangle = G$, and;
\item $xyz=1$. \end{enumerate}
The \textbf{type} of a hyperbolic generating triple $(x,y,z)$ is the triple $(o(x), o(y), o(z))$. If at least two of $o(x), o(y)$ and $o(z)$ are even then we call $(x,y,z)$ an \textbf{even triple}. If $o(x), o(y)$ and $o(z)$ are all odd then we call $(x,y,z)$ an $\textbf{odd triple}$. It is clear that since $z=(xy)^{-1}$, being generated by $x,y$ and $z$ is equivalent to being generated by $x$ and $y$. If the context is clear we may refer to a hyperbolic generating triple simply as a `triple' for brevity and write $(x,y,xy)$ or even just $(x,y)$.\end{definition}

For a positive integer $k$ let $Q_{4k}$ be the dicyclic group of order $4k$ with presentation $$Q_{4k} = \langle p,q \mid p^{2k}=q^4=1,p^q=p^{-1},p^k=q^2 \rangle.$$ Let $G = (H \times H):Q_{4k}$ with the action of $Q_{4k}$ defined as follows. For $(g_1,g_2) \in H \times H$ let $p(g_1,g_2)=(g_1,g_2)$ and $q(g_1,g_2)=(g_2,g_1)$. Then $G^0=H \times H \times \langle p \rangle$ is a subgroup of index $2$ inside $G$.

\begin{theorem} \label{genq} Let $H$ be a perfect finite group. Let $a_1,c_1,a_2,c_2 \in H$ and, as before, define $\nu(a_i,c_i)=o(a_i)o(c_i)o(a_ic_i)$ for $i=1,2$. Assume that \begin{enumerate}
\item the orders of $a_1,c_1$ are even,
\item $\langle a_1, c_1 \rangle = H$,
\item $\langle a_2, c_2 \rangle = H$,
\item $\nu(a_1,c_1)$ is coprime to $\nu(a_2,c_2)$.
\end{enumerate}
Set $k>1$ to be any integer that divides gcd$(o(a_1),o(c_1)), G:=(H \times H):Q_{4k}$, $G^0:=H \times H \times \langle p \rangle$, a:=$(a_1,a_2,p)$ and $c:=(c_1,c_2,p^{-1})$. Then $(G^0;a,c)$ is a mixed Beauville structure on $G$. \end{theorem}

\begin{proof} We verify that the conditions of Definition \ref{mixed} are satisfied. Since $k$ divides $\nu(a_1,c_1)$ it is coprime to $\nu(a_2,c_2)$ so we can clearly generate the elements $(1,a_2,1)$ and $(1,c_2,1)$ giving us the second factor. This also shows that we can generate the elements $a'=(a_1,1,p)$ and $c'=(c_1,1,p^{-1})$. Since $H$ is perfect we can then generate the first factor. Finally, since we can generate $H \times H$ we can clearly generate $\langle p \rangle$, hence we satisfy condition $M1$.

Now let $g \in G \setminus G^0$ and $\gamma \in G^0$. Then $g\gamma$ is of the form $(h_1,h_2,q^ip^j)$ for some $h_1,h_2 \in H$, $i=1,3$ and $1 \leq j \leq 2k$. Then $$(g\gamma)^2 = (h_1h_2,h_2h_1,(q^ip^j)^2)=(h_1h_2,h_2h_1,p^k).$$ For a contradiction, suppose that $(g\gamma)^2 \in \Sigma(a,c)$. Then since $h_1h_2$ has the same order as $h_2h_1$ condition 4 implies that $(g\gamma)^2=(1,1,p^k) \in \Sigma(a,c)$ if and only if $k$ does not divide $o(a)$ or $k$ does not divide $o(c)$. Note that if $(g\gamma)^2$ were a power of $ac$ by construction it would be $1_G$. Since by hypothesis $k$ divides gcd$(o(a_1),o(c_1))$ we satisfy conditions $M2$ and $M3$.

Finally, to show that condition $M4$ is satisfied, suppose $g' \in \Sigma(a,c) \cap \Sigma(a^g,c^g)$ for $g \in G \setminus G^0$. Since conjugation by such an element $g$ interchanges the first two factors of any element, we again have from condition $4$ that $g'$ is of the form $(1,1,p^i)$ for some power of $p$, but from our previous remarks it is clear that $p^i=1_H$ and so $g'=1_G$.
\end{proof}

\begin{remark} In the proof of Theorem \ref{genq} we chose $a:=(a_1,a_2,p)$ and $c:=(c_1,c_2,p^{-1})$ but in principle we could have chosen the third factor of $a$ or $c$ to be $1$ and the third factor in their product $ac$ to be $p$ or $p^{-1}$ as appropriate. If we then required $k$ to divide gcd$(o(a),o(ac))$ or gcd$(o(ac),o(c))$ as necessary this gives rise to further examples of mixed Beauville groups. \end{remark}

%Auxiliary Results bit
We conclude this section with a number of results regarding hyperbolic generating triples for $G \times G$. The number of prime divisors of the order of a group will pose an obvious restriction on finding a mixable Beauville structure and so naturally we would like to know for how large an $n$ we can generate $G^n$ with a triple of a given type. Hall treats this question in much more generality in \cite{hall}, here we simply mention that for a hyperbolic generating triple $(x,y,z)$ of a given type this depends on how many orbits there are of triples of the same type under the action of Aut$(G)$.
%Def
\begin{definition} Let $G$ be a finite group and $(a_1,b_1,c_1)$, $(a_2,b_2,c_2)$ be hyperbolic generating triples for $G$. We call these two triples \textbf{equivalent} if there exists $g \in$ Aut$(G)$ such that $\{a_1^g,b_1^g,c_1^g\} = \{a_2,b_2,c_2\}$. \end{definition}

Since conjugate elements must have the same order the following is immediate.
%Inequivalent
\begin{lemma} \label{inequi} Let $G$ be a finite group and $(a_1,b_1,c_1)$ a hyperbolic generating triple of $G$ of type $(l_1,m_1,n_1)$ and $(a_2,b_2,c_2)$ a hyperbolic generating triple of $G$ of type $(l_2,m_2,n_2)$. If $\{l_1,m_1,n_1\} \neq \{l_2,m_2,n_2\}$ then $(a_1,b_1,c_1)$ and $(a_2,b_2,c_2)$ are inequivalent triples.  \end{lemma}

%Equivalence
\begin{lemma} \label{equi} Let $G$ be a nonabelian finite group and let $(a_1,b_1,c_1)$ and $(a_2,b_2,c_2)$ be equivalent hyperbolic generating triples for $G$ for some $g \in$ Aut$(G)$. Then \begin{enumerate}
\item if $a_1^g=a_2$ then $b_1^g=b_2$ and $c_1^g=c_2$,
\item if $a_1^g=b_2$ then $b_1^g=c_2$ and $c_1^g=a_2$, and;
\item if $a_1^g=c_2$ then $b_1^g=a_2$ and $c_1^g=b_2$. \end{enumerate} \end{lemma}
\begin{proof} Note that in this instance we take $a_ib_ic_i=1_G$ for $i=1,2$. Let $a_1^g=a_2$ and suppose $b_1^g=c_2=(a_2b_2)^{-1}$ and $c_1^g=b_2$. Then $$b_2 = c_1^g=(b_1^{-1}a_1^{-1})^g=c_2^{-1}a_2^{-1}=a_2b_2a_2^{-1}$$ which implies that $G$ is generated by elements that commute, a contradiction since $G$ is nonabelian. The proof is analogous for the remaining two statements. \end{proof}

%cop
\begin{lemma} \label{cop} Let $G$ be a finite group, $(x,y,z)$ a hyperbolic generating triple for $G$ and gcd$(o(x),o(y))=1$. Then $((x,y),(y,x^y),(z,z))$ is a hyperbolic generating triple for $G \times G$.\end{lemma}
\begin{proof} If $(x,y,z)$ is a hyperbolic generating triple of $G$ then so is $(y,x^y,z)$ since $x^y=y^{-1}xy$. Then, since the orders of $x$ and $y$ are coprime we can can produce the elements $(x,1_G), (y,1_G), (1_G,y)$ and $(1_G,x^y)$ which generate $G \times G$. \end{proof}

\begin{remark} The proof of the preceding Lemma naturally generalises to any pair, or indeed triple, of elements in a hyperbolic generating triple whose orders are coprime.\end{remark}

\section{The Alternating groups}%%%%%%%%%%%%%%%%%%%%%%%%%%----ALTERNATING GROUPS----%%%%%%%%%%%%%%%%%%%%%%%%%%
We make heavy use of the following Theorem due to Jordan.

\begin{theorem}[Jordan] Let $G$ be a primitive permutation group of finite degree $n$, containing a cycle of prime length fixing at least three points. Then $G \geqslant A_n$. \end{theorem}
%A6
\begin{lemma} \label{a6} The alternating group $A_6$ and $A_6 \times A_6$ are mixable. \end{lemma}
\begin{proof} For our even triples we take the following elements in the natural representation of $A_6$ $$x_1=(1,2)(3,4,5,6), \; y_1=(1,5,6,4)(2,3), \; \; x_1'=(1,2)(3,4,5,6), \; y_1'=(1,5,6).$$ It can easily be checked in GAP \cite{gap} that $(x_1, y_1, x_1y_1)$ is an even triple for $A_6$ of type $(4,4,4)$ and that $((x_1,x_1'), (y_1,y_1'), (x_1y_1,x_1'y_1'))$ is an even triple for $A_6 \times A_6$ of type $(4,12,12)$.

For our odd triples, let $x_2=(1,2,3,4,5)$, $y_2= x_2^{(1,3,6)}$ and $y_2'=x_2^{(1,2,3,4,6)}$. Then it can be checked that $(x_2,y_2,x_2y_2)$ is an odd triple of type $(5,5,5)$ for $A_6$ and $((x_2,x_2),(y_2,y_2'),(x_2y_2,x_2y_2'))$ is an odd triple for $A_6 \times A_6$ also of type $(5,5,5)$. Therefore we have a mixable Beauville structure of type $(4,4,4;5,5,5)$ for $A_6$ and type $(4,12,12;5,5,5)$ for $A_6 \times A_6$.\end{proof}
%A7
\begin{lemma} The alternating group $A_7$ and $A_7 \times A_7$ are mixable. \end{lemma}
\begin{proof} For our even triples we take the following elements in the natural representation of $A_7$ $$x_1=(1,2)(3,4)(5,6,7), \; y_1=(1,2,3)(4,5)(6,7), \; \; x_1'=(1,6)(2,4,5)(3,7), \; y_1'=(1,6,2)(3,7,4).$$ It can easily be checked in GAP that $(x_1,y_1,x_1y_1)$ is an even triple for $A_7$ of type $(6,6,5)$ and that $((x_1,x_1'),(y_1,y_1'),(x_1y_1,x_1'y_1')$ is an even triple for $A_7 \times A_7$ of type $(6,6,5)$.

For our odd triples let $x_2=(1,2,3,4,5,6,7)$, $y_2=x_2^{(1,3,2)}$ and $y_2'=x_2^{(1,3,2)}$. Again, it can be checked that $(x_1,y_1,x_1y_1)$ is an odd triple of type $(7,7,7)$ for $A_7$ and that $((x_1,x_1'),(y_1,y_1'),(x_1y_1,x_1'y_1'))$ is an odd triple also of type $(7,7,7)$ for $A_7 \times A_7$. Therefore both $A_7$ and $A_7 \times A_7$ admit a mixable Beauville structure of type $(6,6,5;7,7,7)$.\end{proof}
%EvenAlternating
\begin{lemma} \label{alt} The alternating group $A_{2m}$ and $A_{2m} \times A_{2m}$  are mixable for $m \geq 4$. \end{lemma}
\begin{proof} For $m \geq 4$ let $G=A_{2m}$ under its natural representation and consider the elements \begin{align*}
	a_1 &=(1,2)(3,\dots,2m),\\
	b_1 &= a_1^{(1,3,4)}=(1,5,6,\dots,2m,4)(2,3),\\
	a_1b_1 &=(1,3)(2,5,7,\dots,2m-3,2m-1,4,6,8,\dots,2m). \end{align*}
The subgroup $H_1 = \langle a_1,b_1 \rangle$ is clearly transitive and the elements $$a_1^2 = (3,5,\dots,2m-1)(4,6,\dots,2m), \; \; b_1^2 = (1,6,8,\dots,2m)(5,7,\dots,2m-1,4),$$fix the point 2 and act transitively on the remaining points. Finally, $a_1^2 b_1^{-2}=(1,2m,2m-1,3,4)$ is a prime cycle fixing at least 3 points for all $m$ and so by Jordan's Theorem $H_1 = G$. This gives us our first hyperbolic generating triple of type $(2m-2,2m-2,2m-2)$ for $G$. For our second, we show that there is a similar triple which is inequivalent to the first under the action of Aut$(G) = S_{2m}$. Consider the elements \begin{align*}
	a_1' &=(1,2)(3,\dots,2m),\\
	b_1' &= a_1'^{(1,4,3)}=(1,3,5,6,\dots,2m)(2,4),\\
	a_1'b_1' &=(1,4,6,\dots,2m,5,7,\dots,2m-1)(2,3) \end{align*}
and note that $a_1'=a_1$. For the same argument as before we have that $\langle a_1',b_1' \rangle = G$. Now suppose that $(a_1,b_1,a_1b_1)$ is equivalent to $(a_1',b_1',a_1'b_1')$ for some $g \in$ Aut$(G)$. If $a_1^g = a_1'$ then by Lemma \ref{equi} we have that $b_1^g=b_1'$ and $(a_1b_1)^g=a_1'b_1'$. Then $(1,2)^g=(1,2)$ and $(2,3)^g=(2,4)$ but these are incompatible with $(1,3)^g=(2,3)$ since for the former to hold $3$ must map to $4$ which is incompatible with the latter. Now suppose $a_1^g=b_1'$ implying $b_1^g=a_1'b_1'$ and $a_1b_1=a_1$. Then similarly we have $(1,2)^g=(2,4)$ and $(2,3)^g=(2,3)$ forcing $g$ to map $1$ to $4$ which is incompatible with requiring that $(1,3)^g=(1,2)$. Finally, if $a_1^g=a_1'b_1'$, forcing $b_1^g=a_1'$ and $a_1b_1=b_1'$, we get that $(1,2)^g=(2,3)$ and $(2,3)^g=(12)$ and we find this is incompatible with $(1,3)^g=(2,4)$. Hence these two hyperbolic generating triples are inequivalent under the action of the automorphism group of $G$ and so $((a_1,a_1'),(b_1,b_1'),(a_1b_1,a_1'b_1'))$ is a hyperbolic generating triple for $G \times G$ of type $(2m-2,2m-2,2m-2)$.

For our first odd triple consider the elements \begin{align*}
	a_2 &= (1,2,\dots,2m-1),\\
	b_2 &= a_2^{(1,2m,3)} = (1,4,5,\dots,2m-1,2m,2), \\
	a_2b_2 &=(2,3,5,7,\dots,2m-1,4,6,8,\dots,2m-2,2m) \end{align*}
and let $H_2 = \langle a_2,b_2 \rangle$. We clearly have transitivity and 2-transitivity, hence $H_2$ is primitive. Since $a_2b_2^{-1}=(1,2m,2m-1,2,3)$ is a prime cycle fixing at least 3 points for all $m$, again we can apply Jordan's Theorem and we have that $H_2 = G$. Then $(a_1,b_1,a_2,b_2)$ is a mixable Beauville structure on $G$ of type $$(2m-2,2m-2,2m-2;2m-1,2m-1,2m-1).$$ For our second odd triple consider the elements \begin{align*}
	a_2' &= (1,2,\dots,2m-1),\\
	b_2' &= a_2'^{(1,3,2m)} = (2,2m,4,\dots,2m-1,3), \\
	a_2'b_2' &=(1,2m,4,6,\dots,2m-2,3,5,\dots,2m-1)\end{align*}
and note that $a_2'=a_2$. It follows that $\langle a_2,b_2 \rangle = G$ from a similar argument as before and so it remains to show that $\langle (a_2,a_2'),(b_2,b_2')\rangle = G \times G$. Let $g \in$ Aut$(G)$ and suppose that $a_2^g=a_2'$. Then by Lemma \ref{equi} and inspection of the fixed points of these triples we have $g$ fixes $2m$ and maps $3 \to 1$ and $1 \to 2$ which is incompatible with $a_2^g=a_2'$. Similarly, if $a_2^g=b_2'$ then $2m \to 1$, $3 \to 2$ and $1 \to 2m$ but from $b_2^g=a_2'b_2'$ $g$ must then map $3 \to 5$, a contradiction. Finally, if $a_2^g=a_2'b_2'$ we get the mappings $2m \to 2$, $3 \to 2m$ and $1 \to 1$. But since $a_2^g=a_2'b_2'$, $g$ must then also map $3 \to 4$, a final contradiction. Then $(a_2,b_2,a_2b_2)$ and $(a_2',b_2',a_2'b_2')$ are inequivalent hyperbolic generating triples both of type $(2m-1,2m-1,2m-1)$ on $G$ and so we get a mixable Beauville structure on $G \times G$ of type $$(2m-2,2m-2,2m-2;2m-1,2m-1,2m-1).$$
\end{proof}
%OddAlternating
\begin{lemma} The alternating group $A_{2m+1}$ and $A_{2m+1} \times A_{2m+1}$ are mixable for $m \geq 4$. \end{lemma}
\begin{proof} For $m \geq 4$ let $G=A_{2m+1}$ under its natural representation and consider the elements \begin{align*}
	a_1 &= (1,2)(3,4)(5,\dots,2m+1),\\
	b_1 &=(1,2,\dots,2m-3)(2m-2,2m-1)(2m,2m+1),\\
	a_1b_1 &= (1,3,5,\dots,2m-1,2m+1,6,8,\dots,2m-4).\end{align*}
The subgroup $G_1 = \langle a_1,b_1 \rangle$ is clearly transitive and the elements \begin{align*}
	b_1a_1^2b_1^{-1} &= (4,6,8,\dots,2m,5,7,\dots,2m-5,2m-1,2m+1),\\
	a_1b_1^2a_1^{-1} &= (2,4,2m+1,6,8,\dots,2m-4,1,3,\dots,2m-5)\end{align*}
both fix the point $2m-3$ and act transitively on the remaining points; hence $G_1$ acts primitively. Finally, the element $a_1b_1^{-1} = (2,2m-3,2m-1,2m+1,4)$ is a prime cycle fixing at least 3 points for all $m \geq 4$ and so by Jordan's Theorem $G_1 =G$. This gives our first hyperbolic generating triple of type $(2(2m-3),2(2m-3),2m-3)$ for $G$. For our second even triple, we manipulate the first in the following way. Let \begin{align*}
	a_1' &=(1,2m-4,\dots,6,2m+1,2m-1,\dots,3),\\
	b_1' &=(1,2)(3,4)(5,\dots,2m+1),\\
	a_1'b_1' &=(1,2m-3,\dots,2)(2m-2,2m-1)(2m,2m+1).\end{align*}
Since $b_1'=a_1$ and $a_1'b_1'=b_1^{-1}$ it is clear that $\langle a_1',b_1' \rangle=G$. Note also that $a_1'=b_1^{-1}a_1^{-1}$. To see that $(a_1,b_1,a_1b_1)$ and $(a_1',b_1',a_1'b_1')$ are inequivalent, suppose for a contradiction there exists $g \in$ Aut$(G)$ for which these triples are equivalent. Since conjugation preserves cycle type it must be that $(a_1b_1)^g=a_1'$ which, by Lemma \ref{equi}, implies that $a_1^g=b_1'$ and $b_1^g=a_1'b_1'$. This give the equality $$a_1b_1^{-1}=b_1'(a_1'b_1')=a_1^gb_1^g=(a_1b_1)^g=a_1'=b_1^{-1}a_1,$$ a contradiction since otherwise $G$ would be abelian. Then, $((a_1,a_1'),(b_1,b_1'))$ is an even triple for $G \times G$ of type $(2(2m-3),2(2m-3),2(2m-3))$.

For our first odd triple consider the elements \begin{align*}
	a_2 &= (1,2, \dots, 2m+1),\\
	b_2 &= a_2^{(1,2,3)} = (1,4,5,\dots,2m,2m+1,2,3),\\
	a_2b_2 &= (1,3,5,\dots,2m-1,2m+1,4,6,\dots,2m-2,2m,2).\end{align*}
The subgroup $G_2 = \langle a_2,b_2 \rangle$ is clearly transitive while the elements $b_2^{-1}a_2^2=(1,5,6,\dots,2m+1)(3,4)$ and $a_2b_2^{-1}=(1,2m+1,3)$ fix the point 2 and act transitively on the remaining points. Hence $G_2$ is primitive, with a prime cycle fixing at least 3 points, then by Jordan's Theorem $G_2=A_{2m+1}$. This gives us an odd triple of type $(2m+1,2m+1,2m+1)$ for $G$ and so since we have gcd$(2(2m-3),2m+1)=1$ it follows that $(a_1,b_1,a_2,b_2)$ is a mixable Beauville structure for $A_{2m+1}$ of type $$(2(2m-3),2(2m-3),2m-3;2m+1,2m+1,2m+1).$$ For our second odd triple consider the cycles \begin{align*}
	x_2 &= (1,2, \dots, 2m-1),\\
	y_2 &= x_2^{(1,2m,2,2m+1,3)} = (1,4,5,\dots,2m-1,2m,2m+1),\\
	x_2y_2 &= (1,2,3,5,\dots,2m-1,4,6,\dots,2m,2m+1)\end{align*}
and let $H_2 = \langle x_2,y_2 \rangle$. We clearly have transitivity while the elements $[x_2,y_2] = (1,2m,4,5,2)$ and $x_2[x_2,y_2] = (2,3,5,6,\dots,2m-1,2m,4)$ show that $H_2$ acts transitively on the stabiliser of the point $2m+1$ and contains a prime cycle fixing at least 3 points for all $m$. Then by Jordan's Theorem $H_2 = G$ and this gives us a second odd triple of type $(2m-1,2m-1,2m-1)$. Since it is clear that $2(2m-3)$ is coprime to both $2m-1$ and $2m+1$ we then have a mixable Beauville structure on $G \times G$ of type $$(2(2m-3),2(2m-3),2(2m-3);4m^2-1,4m^2-1,4m^2-1).$$\end{proof}

\section{The groups of Lie type}%%%%%%%%%%%%%%%%%%%%%%%%%%----GROUPS OF LIE TYPE----%%%%%%%%%%%%%%%%%%%%%%%%%%

We make use of theorems due to Zsigmondy, generalising a theorem of Bang, and Gow which we include here for reference. Throughout this section $q=p^e$ will denote a prime power for a natural number $e \geq 1$.

\begin{theorem}[Zsigmondy \cite{zs} or Bang \cite{bang}, as appropriate] \label{zsig} For any positive integer $a > 1$ and $n > 1$ there is a prime number that divides $a^n-1$ and does not divide $a^k-1$ for any positive integer $k < n$, with the following exceptions: \begin{enumerate}
\item $a = 2$ and $n=6$; and
\item $a+1$ is a power of $2$, and $n = 2$. \end{enumerate} \end{theorem}

We denote a prime with such a property $\Phi_n(a)$.

\begin{remark} The case where $a=2$, $n > 1$ and not equal to $6$ was proven by Bang in \cite{bang}. The general case was proven by Zsigmondy in \cite{zs}. Hereafter we shall refer to this as Zsigmondy's Theorem.  A more recent account of a proof is given by L\"{u}nburg in \cite{Lub}. \end{remark}

\begin{definition} Let $G$ be a group of Lie type defined over a field of characteristic $p>0$, prime. A \textbf{semisimple} element is one whose order is coprime to $p$. A semisimple element is \textbf{regular} if $p$ does not divide the order of its centraliser in $G$. \end{definition}

\begin{theorem}[Gow \cite{gow}] \label{gow} Let $G$ be a finite simple group of Lie type of characteristic $p$, and let $g$ be a non-identity semisimple element in $G$. Let $L_1$ and $L_2$ be any conjugacy classes of $G$ consisting of regular semisimple elements. Then $g$ is expressible as a product $xy$, where $x \in L_1$ and $y \in L_2$. \end{theorem}

\begin{remark} A slight generalisation of this result to quasisimple  groups appears in \cite[Theorem 2.6]{fmp}. \end{remark}

%%%%%%%%%%%%%%%%%%%%%%%%%%----PSL2
\subsection{Projective Special Linear groups $PSL_2(q) \cong A_1(q)$}
The projective special linear groups $PSL_2(q)$ are defined over fields of order $q$ and have order $q(q+1)(q-1)/k$ where $k=$gcd$(2,q+1)$. Their maximal subgroups are listed in \cite{fj}.
%PSL_2(7)
\begin{lemma} \label{l27} Let $G=PSL_2(7)$. Then $G^n$ admits a mixable Beauville structure if and only if $n=1,2$. \end{lemma}
\begin{proof} The maximal subgroups of $G$ are known \cite[p.2]{ATLAS} and these are subgroups isomorphic to $S_4$ or point stabilisers in the natural representation of $G$ on 8 points. Hyperbolic generating triples cannot have type $(3,3,3)$, since $\frac{1}{3} + \frac{1}{3} + \frac{1}{3} \nless 1$, and similarly for types $(2,2,2)$, $(2,2,4)$ or $(2,4,4)$. The number of hyperbolic generating triples of type $(7,7,7)$ can be computed using GAP, but since it is equal to the order of Aut$(G)$ we see from \cite{hall} that there is no triple of type $(7,7,7)$ for $G^n$ when $n>1$. Triples of type $(4,4,4)$ exist and any such triple generates $G$ since elements of order 4 are not contained in point stabilisers and inside a subgroup isomorphic to $S_4$ the product of three elements of order $4$ cannot be equal to the identity. We can compute the number of such triples from the structure constants and since this is twice the order of Aut$(G)$ we have that there exists a hyperbolic generating triple of type $(4,4,4)$ on $G$ and on $G \times G$. We then see that this is the maximum number of direct copies of $G$ for which there exists a mixable Beauville structure.

For our odd triple we then take a triple of type $(7,7,3)$ which can be shown to exist by computing their structure constants and are seen to generate $G$ since if they were to belong to a maximal subgroup then the product of two elements of order $7$ would again have order $7$. This gives a mixable Beauville structure of type $(4,4,4;7,7,3)$ on $G$. Finally, we then have mixable Beauville structures on $G \times G$ of type $(4,4,4;7,7,21)$ by Lemma \ref{inequi} or alternatively of type $(4,4,4;7,21,21)$ by Lemma \ref{cop}.\end{proof}
%PSL_2(8)
\begin{lemma} \label{l28} Let $G=PSL_2(8)$. Then $G \times G$ admits a mixable Beauville structure.\end{lemma}
\begin{proof} It can easily be checked in GAP that for $G$ there exists hyperbolic generating triples of types $(2,7,7)$, $(3,3,9)$ and $(3,9,9)$. The two odd triples are inequivalent by Lemma \ref{inequi} and by Lemma \ref{cop} we have that there exists a mixable Beauville structure of type $(14,14,7;3,9,9)$ on $G \times G$.\end{proof}
%PSL_2(9)
\begin{lemma} \label{l29} Let $G=PSL_2(9)$. Then both $G$ and $G \times G$ admit a mixable Beauville structure. \end{lemma}
\begin{proof} This follows directly from the exceptional isomorphism $PSL_2(9) \cong A_6$ and Lemma \ref{a6}. \end{proof}

We make use of the following Lemmas:
\begin{lemma} \label{phil} Let $G=PSL_2(q)$ for $q=7, 8$ or $q \geq 11$. Let $k=$ gcd$(2,q+1)$ and $\phi(n)$ be Euler's totient function. Then, under the action of Aut$(G) = P \Gamma L_2(q)$ the number of conjugacy classes of elements of order $\frac{q+1}{k}$ in $G$ is $\phi(\frac{q+1}{k})/2e$.\end{lemma}
\begin{proof} Elements of order $\frac{q+1}{k}$ are conjugate to their inverse so there are $\phi(\frac{q+1}{k})/2$ conjugacy classes of elements of order $\frac{q+1}{k}$ in $PSL_2(q)$. The only outer automorphisms of $G$ come from the diagonal automorphisms and the field automorphisms, but since diagonal automorphisms do not fuse conjugacy classes of semisimple elements we examine the field automorphisms. These come from the action of the Frobenius automorphism on the elements of the field $\FF_q$ sending each entry of the matrix to its $p$-th power. The only fixed points of this action are the elements of the prime subfield $\FF_p$ and so, since the entries on the diagonal of the elements of order $\frac{q+1}{k}$ are not both contained in the prime subfield we have that the orbit under this action has length $e$, the order of the Frobenius automorphism. We then get $e$ conjugacy classes of elements of order $\frac{q+1}{k}$ inside of $G$ fusing under this action. Hence under the action of the full automorphism group there are $\phi(\frac{q+1}{k})/2e$ conjugacy classes of elements of order $\frac{q+1}{k}$. \end{proof}

\begin{lemma} \label{2e+1} For a prime power, $q=p^e \geq 13$, $q \neq 27$, let $q^+=\frac{q+1}{k}$ where $k=$ gcd$(2,q+1)$. Then $\frac{\phi(q^+)}{2e} > 1$ where $\phi(n)$ is Euler's totient function.  \end{lemma}
\begin{proof} Let $S=\{p^i,q^+-p^i \mid 0 \leq i \leq e-1 \}$ be a set of $2e$ positive integers less than and coprime to $q^+$ and whose elements are distinct when $q \geq 13$. To this set we add $k$ which will depend on $q$. When $p=2$ we let $k=7$ since for all $e>3$, $7 \notin S$ and gcd$(7,q^+)=1$. When $p=3$ we let $k=11$, then for $e>3$, $11 \notin S$ and gcd$(11,q^+)=1$. Now consider the cases $q \equiv \pm 1$ mod $4$ for $p \neq 2,3$. When $q \equiv 1$ mod $4$, let $k=q^+-2$. Since $p \neq 2$ we have $k \notin S$ and since $q^+$ is odd when $q \equiv 1$ mod $4$ we have gcd$(k,q^+)=1$. Finally, when $q \equiv 3$ mod $4$ then $e$ must be odd. When $e > 2$ then $k=\frac{p-1}{2} \notin S$ and is coprime to $q^+$. When $e=1$, $q^+=2^im$ where $i>0$ and $m$ is odd. Then $\phi(q^+)=\phi(2^i)\phi(m)=2^{i-1}\phi(m)>2$ since $p>11$. This completes the proof. \end{proof}

\begin{lemma} \label{psl22} Let $G$ be the projective special linear group $PSL_2(q)$ where $q \geq 7$. Then, \begin{enumerate}
\item there is a mixable Beauville structure for $G \times G$, and;
\item when $p \neq 2$ there is also a mixable Beauville structure for $G$.\end{enumerate} \end{lemma}
\begin{proof} In light of Lemmas \ref{l27}--\ref{l29} we can assume that $q \geq 11$. Define $q^+= \frac{q+1}{k}$, where $k=$ gcd$(2,q+1)$, and similarly for $q^-$. Jones proves in \cite{JonesChar2} that hyperbolic generating triples of type $(p,q^-,q^-)$ exist for $G$ when $q \geq 11$ and since gcd$(p,q^-)=1$ we immediately have, by Lemma \ref{cop}, a hyperbolic generating triple for $G \times G$. We proceed to show that there exists a hyperbolic generating triple $(x,y,z)$ for $G$ of type $(q^+,q^+,q^+)$ and note that both $p$ and $q^-$ are coprime to $q^+$. The only maximal subgroups containing elements of order $q^+$ are the dihedral groups of order $2q^+$ which we denote by $D_{q^+}$. By Gow's Theorem, for a conjugacy class, $C$, of elements of order $q^+$ there exist $x,y,z \in C$ such that $xyz=1$. Since inside $D_{q^+}$ any conjugacy class of elements of order $q^+$ contains only two elements, $x,y$ and $z$ can not all be contained in the same maximal subgroup of $G$. Hence $(x,y,z)$ is a hyperbolic generating triple for $G$ of type $(q^+,q^+,q^+)$. When the number of conjugacy classes of elements of order $q^+$ in $G$ under the action of Aut$(G)$ is strictly greater than $1$ we can apply Gow's Theorem a second time to give a hyperbolic generating triple of type $(q^+,q^+,q^+)$ for $G \times G$. This follows from Lemmas \ref{phil} and \ref{2e+1} with the exceptions of $q=11$ or $27$. For $G=PSL_2(11)$ we have that a triple of type $(p,q^-,q^-)$ exists by \cite{JonesChar2} or alternatively the words $ab$ and $[a,b]$ in the standard generators for $G$ \cite{brauer} give an odd triple of type $(11,5,5)$. In both cases we have, by Lemma \ref{cop}, an odd triple of type or $(55,55,5)$ for $G \times G$. For our even triple, the structure constants for the number of triples of type $(6,6,6)$ can be computed and is seen to be twice the order of Aut$(G)$ and so we have an even triple for $G$ and $G \times G$. For $G=PSL_2(27)$ we take the words in the standard generators \cite{brauer} $(ab)^2(abb)^2$, $a^{b^2}$ which give an even triple of type $(2,14,7)$ and the words $b^2,b^a$ which give an odd triple of type $(3,3,13)$. Again, by Lemma \ref{cop}, these give a mixable Beauville structure on $G \times G$. Finally, we remark that when $q \equiv \pm 1$ mod $4$ we have that $q^-$ and $q^+$ have opposite parity and this determines the parity of our triples. When $q \equiv 1$ mod $4$, $(p,q^-,q^-)$ becomes our even triple, $(q^+,q^+,q^+)$ our odd triple and vice versa when $q \equiv 3$ mod $4$. \end{proof}

%%%%%%%%%%%%%%%%%%%%%%%%%%----PSU3
\subsection{Projective Special Unitary groups $PSU_3(q) \cong$ $^2A_2(q)$}
The projective special unitary groups $PSU_3(q)$ are defined over fields of order $q^2$ and have order $q^3(q^3+1)(q^2-1)/d$ where $d=(3,q+1)$. Their maximal subgroups can be found in \cite{psu3} and we refer to the character table and notation in \cite{ssf}.

\begin{lemma} \label{psut} \label{t} Let $G=PSU_3(q)$ for $q = 4$ or $q \geq 7$. Let $d=gcd(3,q+1)$ and $t' = \frac{q^2-q+1}{d}$. Then there exists a hyperbolic generating triple of type $(t', t', t')$ for $G$. \end{lemma}
\begin{proof} Let $G, d$ and $t'$ be as in the hypothesis. Let $C$ be a conjugacy class of elements of order $t'$ in $G$ and for $g \in C$, let $T = \langle g \rangle$. When $q \geq 7$ the unique maximal subgroup of $G$ containing $g$ is $N_G(T)$ of order $3t'$. Since gcd$(6,t')=1$ and since $T$ is by definition normal in $N_G(T)$ the Sylow $p$-subgroup in $N_G(T)$ for any prime $p \vert t'$ is contained in $T$ and is therefore unique. Hence for all $x \in N_G(T) \setminus T$ the order of $x$ is 3. Then for $g \in C$ since $g$ is conjugate to $g^{-q}$ and $g^{q^2}$ \cite{ssf} we have $C \cap N_G(T) = \{g, g^{-q}, g^{q^2} \}$. This partitions $C$ into $\vert C \vert /3$ disjoint triples. Using the structure constants obtained from the character table of $G$ we show that it is possible to find a hyperbolic generating triple of $G$ entirely contained within $C$. Our method is to count the total number of triples $(x,y,z) \in C \times C \times C$ such that $xyz=1$, which we denote $n(C,C,C)$, and show that there exists at least one such triple where $x,y,z$ come from distinct maximal subgroups. Let $S$ be a triple of the form $\{g,g^{-q},g^{q^2}\} \subset C$, then for $s_1, s_2, s_3 \in S, s_1s_2s_3=1$ if and only if all three elements are distinct. Therefore the contribution to $n(C,C,C)$ from triples contained within a single maximal subgroup of $G$ is $2 \lvert C \rvert$. Using the formula for structure constants as found in \cite{frob} we have that $$n(C,C,C) = \frac{\lvert C \rvert^3}{\lvert G \rvert} \left( 1- \frac{1}{q(q-1)} -\frac{1}{q^3} -\sum_1^{t'-1} \frac{(\zeta_{t'}^u + \zeta_{t'}^{-uq} + \zeta_{t'}^{uq^2})^3}{3(q+1)^2(q-1)} \right)$$ where $\zeta_{t'}$ is a primitive $t'$-th root of unity. From the triangle inequality we have $\lvert (\zeta_{t'}^u + \zeta_{t'}^{-uq} + \zeta_{t'}^{uq^2})\rvert^3 \leq 27$ so we can bound $n(C,C,C)$ from below and for $q \geq 8$ the following inequality holds $$n(C,C,C) \geq \frac{\lvert C \rvert^3}{\lvert G \rvert} \left( 1- \frac{1}{q(q-1)} -\frac{1}{q^3} - \frac{9(t'-1)}{(q+1)^2(q-1)} \right) > 2\lvert C \rvert.$$ For $q = 4$ or $7$ direct computation of the structure constants show that we can indeed find a hyperbolic generating triple of the desired type. \end{proof}

\begin{lemma} \label{psu3} Let $G=PSU_3(q)$ for $q = 4$ or $ q\geq 7$. Let $c=gcd(3,q^2-1)$, $d=gcd(3,q+1)$ and $t'=(q^2-q+1)/d$. Then \begin{enumerate}
\item for $p=2$ there exists a mixable Beauville structure on $G$ of type $$\left(2,4,\frac{q^2-1}{c};t', t', t'\right),$$
\item for $p \neq 2$ there exists a mixable Beauville structure on $G$ of type $$\left(p,q+1,\frac{q^2-1}{d};t',t',t' \right).$$
\end{enumerate}
\end{lemma}
\begin{proof} The existence of hyperbolic generating triples of type $(t',t',t')$ in both even and odd characteristic is given in Lemma \ref{psut} and so we turn to the even triples. When $p=2$ we have the existence of our even triples from \cite[Lemma 4.20 and Theorem 4.22]{fmp} and so we now assume that $G=PSU_3(q)$ where $p$ is odd, $q \geq 7$, $r=q+1$ and $s=q-1$.

From the list of maximal subgroups of $G$ elements of order $rs/d$ exist and can belong to subgroups corresponding to stabilisers of isotropic points, stabilisers of non-isotropic points and possibly one of the maximal subgroups of a fixed order which can occur is $PSU_3(q)$ for certain $q$. Stabilisers of isotropic points have order $q^3r/d$ whereas stabilisers of non-isotropic points have order $qr^2s/d$. There exist $1+d$ conjugacy classes of elements of order $p$ in $G$ which are as follows. The unique conjugacy class, $C_2$, of elements whose centralisers have order $q^3r/d$; and $d$ conjugacy classes, $C_3^{(l)}$ for $0 \leq l \leq d-1$, of elements whose centralisers have order $q^2$. Since an element of order $p$ which stabilises a non-isotropic point belongs to a subgroup of $G$ isomorphic to $SL_2(q)$, the order of its centraliser in $G$ must be a multiple of $2p$, hence must belong to $C_2$. In particular, elements of $C_3^0$ do not belong to stabilisers of non-isotropic points. There exists a conjugacy class, $C_6$, of elements of order $r$ whose centralisers have order $r^2/d$. An element of order $r$ contained in the stabiliser of an isotropic point must be contained in a cyclic subgroup of order $rs/d$, hence $rs/d$ must divide the order of its centraliser and so elements of $C_6$ are not contained in the stabilisers of isotropic points. Let $C_7$ be a conjugacy class of elements of order $rs/d$, then, any triple of elements $(x,y,z) \in C_3^0 \times C_6 \times C_7$, such that $xyz=1$, will not be entirely contained within the stabiliser of an isotropic or non-isotropic point. Let $n(C_3^0,C_6,C_7)$ be the number of such triples, then using the structure constant formula from \cite{frob} and the character table for $G$ we have $$n(C_3^0,C_6,C_7) = \frac{\vert C_3^0 \vert \vert C_6 \vert \vert C_7\vert}{\vert G \vert}\left( 1 + \sum_{u=1}^{\frac{r}{d}-1} \frac{\epsilon^{3u}(\epsilon^{3u} + \epsilon^{6u} + \epsilon^{(r-3)u})}{t}\right)$$ where $\epsilon$ is a primitive $r$-th root of unity. Using the triangle inequality we can bound the absolute value of the summation by $\frac{3q}{q^2-q+1}$ which, for $q \geq 7$, is strictly less than 1. In order to show that such an $(x,y,z)$ is not contained in any of the possible maximal subgroups of order 36, 72, 168, 216, 360, 720 or 2520, notice that the subgroup generated by $(x,y,z)$ has order divisible by $n=p(q^2-1)/d$, hence this can only occur when $p=$ 3, 5 or 7. The only cases where $n \leq 2520$ and divides one of the possible subgroup orders are the cases $q=7$ or $9$, but none of these subgroups contain elements of order 48 or 80 so we see that this is indeed an even triple for $G$. Finally, we must show that gcd$(\frac{4rs}{c},t')=1$ when $p$ is even and gcd$(\frac{prs}{d},t')=1$ when $p$ is odd. For all $p$ it is clear that gcd$(p,t')=1$ and so it suffices to show that gcd$(rs,t')=1$. We have that $t'd-s=q^2$ so $t'$ is coprime to $s$ and since $r^2-t'd=3q$ and $t'$ is coprime to 3 we have that $t'$ is coprime to $r$.
\end{proof}

In order to extend this to a mixable Beauville structure on $G \times G$ where $G=PSU_3(q)$ we will need the following Lemma:

\begin{lemma} \label{phiu} Let $G=PSU_3(q)$ for $q \geq 3$ and $d=(3,q+1)$. Then the number of conjugacy classes of elements of order $t'=(q^2-q+1)/d$ in $G$ under the action of Aut$(G)$ is $\phi(t')/6e$ where $\phi(n)$ is Euler's totient function. \end{lemma}
\begin{proof} Let $x \in G$ have order $t'$, then $x$ is conjugate to $x^{-q}$ and $x^{q^2}$ and so the number of conjugacy classes of order $t'$ in $G$ is $\phi(t')/3$. Then since the field automorphism has order $2e$ and the diagonal entries of an element of order $t'$ are not all contained in the prime subfield its orbit has length $2e$. This gives the desired result. \end{proof}

\begin{lemma} \label{u3u3} Let $G$ be the projective special unitary group $PSU_3(q)$ for $q=7$ or $q \geq 9$. Let $c=$ gcd$(3,q^2-1)$, $d=$ gcd$(3,q+1)$ and $t'=(q^2-q+1)/d$. Then \begin{enumerate}
\item for $p=2$ there exists a mixable Beauville structure on $G \times G$ of type $$\left(2\frac{q^2-1}{c},2\frac{q^2-1}{c},4;t',t',t' \right),$$ and;
\item for $p \neq 2$ there exists a mixable Beauville structure on $G \times G$ of type $$\left(p(q+1),p(q+1),p\frac{q^2-1}{d};t',t',t' \right).$$
\end{enumerate}
\end{lemma}
\begin{proof} Let the conditions of the hypothesis be satisfied with $p=2$. Then as in the proof of Lemma \ref{psu3} there exists a hyperbolic generating triple of type $(2,4,\frac{q^2-1}{c})$ on $G$ and by Lemma \ref{cop} this yields an even triple of type $(2\frac{q^2-1}{c},2\frac{q^2-1}{c},4)$ on $G \times G$. Similarly, for $p \neq 2$ by Lemma \ref{psu3} there exists a hyperbolic generating triple of type $(p,q+1,\frac{q^2-1}{d})$ on $G$, which by Lemma \ref{cop} yields an even triple of type $(p(q+1),p(q+1),\frac{q^2-1}{d})$ on $ G\times G$.

By Lemma \ref{psut} there exist hyperbolic generating triples of type $(t',t',t')$ for all $p$ and by Lemma \ref{phiu} we need only show that $\phi(t') > 6e$. For $q=7,9$ it can be verified directly that $\phi(t')>6e$ so we can assume $q \geq 11$. In the case $d=1$ we have $2^3 < p^e-1$ so $2^f < p^{2f}-p^f+1$ for all $0 \leq f \leq 6e$ and we have our inequality. In the case $d=3$ since $2^33 < p(p-1)$ we have $2^f$ for $1 \leq f \leq 3e$. Similarly, since $3^2 < p-1$ we have the terms $3^{f-1}$ for $1 \leq f \leq 4e$ giving us $7e$ terms in total, as was to be shown. \end{proof}

\begin{lemma} \label{U33458} The projective special unitary group $PSU_3(q)$ and $PSU_3(q) \times PSU_3(q)$ admit a mixable Beauville structure for $q \geq 3$.\end{lemma}
\begin{proof}
In light of the preceding Lemmas in this section it remains only to check the cases $PSU_3(3)$, $PSU_3(5)$ and $PSU_3(q) \times PSU_3(q)$ for $q=3,4,5$ and $8$. We present words in the standard generators \cite{sg, brauer} that can be easily checked to give suitable triples for $G$ which, by Lemma \ref{cop}, give mixable Beauville structures for these cases. For $G=PSU_3(3)$ let $$a_1=[a,b^2], \; b_1=[a,b^2]^b, \; a_2=(babab^2)^3, \; b_2=[a,b^2]^b \in G.$$ It can be checked that $G = \langle a_1,b_1 \rangle = \langle a_2,b_2 \rangle$ where both hyperbolic generating triples have type $(4,4,8)$ but in the former triple both elements of order $4$ come from the conjugacy class $4C$, whereas in the latter, $a_2 \in 4AB, b_2 \in 4C$. Since these two triples are then inequivalent under the action of Aut$(G)$ we have that $(a_1,a_2), (b_1,b_2) \in G \times G$ yields a hyperbolic generating triple of type $(4,4,8)$. For the remaining cases we present in Table \ref{u3table} words in the standard generators for $G$ which, by Lemma \ref{cop}, give a mixable Beauville structure on $G \times G$ and $G$ where necessary.
\begin{table} \centering
\begin{tabular}{l c c c c c} \hline
$q$		&$x_1$			&$y_1$			&$x_2$			&$y_2$				&Type\\ \hline
$3$		&$a_1,a_2$		&$b_1,b_2$		&$ab$			&$ba$				&$(4,4,8;7,7,3)$\\
$4$		&$a$			&$(ab)^2$		&$b$			&$[b,a]$				&$(2,13,13;3,5,3)$\\
$5$		&$a$			&$ab^2$			&$ab$			&$b^3ab^3$			&$(3,8,8;7,7,5)$\\
$8$		&$a$			&$(ab)^2$		&$[a,b]$			&$[a,b]^{babab}$		&$(2,19,19;9,9,7)$\\ \hline
\end{tabular}
\caption{Words in the standard generators \cite{brauer} $a$ and $b$ or as otherwise specified in Lemma \ref{U33458} for $G = PSU_3(q)$.} \label{u3table}
\end{table}
\end{proof}

%%%%%%%%%%%%%%%%%%%%%%%%%%----Suzuki
\subsection{The Suzuki groups $^2B_2(2^{2n+1}) \cong Sz(2^{2n+1})$}
The Suzuki groups $^2B_2(q)$ are defined over fields of order $q=2^{2n+1}$ for $n \geq 0$ and have order $q^2(q-1)(q^2+1)$. They are simple for $q>2$ and their maximal subgroups can be found in \cite{raw}.

\begin{lemma} \label{sz} Let $G$ be the Suzuki group $^2B_2(q)$ for $q>2$. Then \begin{enumerate}
\item $G$ admits a mixable Beauville structure of type $(2,4,5;q-1,n,n)$, and;
\item $G \times G$ admits a mixable Beauville structure of type $(4,10,10;n(q-1),n(q-1),n)$ \end{enumerate}
where $n = q \pm \sqrt{2q} +1$, whichever is coprime to 5.\end{lemma}
\begin{proof} In the proof of \cite[Theorem 6.2]{fj} Fuertes and Jones prove that there exist hyperbolic generating triples for $G$ of types $(2,4,5)$ and $(q-1,n,n)$. It is clear that gcd$(10,n)=$ gcd$(10,q-1)=1$. Then, by Lemma \ref{cop}, we need only show that gcd$(q-1,n)=1$. If $q-1$ and $n$ share a common factor, then so do $q^2-1$ and $q^2+1$ and similarly their difference. Hence gcd$(q-1,n)$ divides 2, but since $q-1$ is odd we have gcd$(q-1,n)=1$ as was to be shown. \end{proof}

%%%%%%%%%%%%%%%%%%%%%%%%%%----Small Ree
\subsection{The small Ree groups $^2G_2(3^{2n+1}) \cong R(3^{2n+1})$}

The small Ree groups are defined over fields of order $q=3^{2n+1}$ for $n \geq 0$ and have order $q^3(q^3+1)(q-1)$. They are simple for $q>3$ and their maximal subgroups can be found in \cite{2g2} or \cite{raw}.

\begin{lemma} \label{ree} The small Ree groups $^2G_2(q)$ for $q>3$ admit a mixable Beauville structure of type $$\left(\frac{q+1}{2},\frac{q+1}{2},q+\sqrt{3q}+1;\frac{q-1}{2},\frac{q-1}{2},q-\sqrt{3q}+1\right).$$ \end{lemma}
\begin{proof} Let $G= {}^2G_2(q)$ for $q>3$ and for convenience we let $n^+ =q+\sqrt{3q}+1$ and $n^- =q-\sqrt{3q}+1$. Ward's analysis of the small Ree groups \cite{ward} shows that elements of orders $\frac{q+1}{2}$, $\frac{q-1}{2}$, $n^+$ and $n^-$ exist and that the order of their centralisers are $q+1$, $q-1$, $n^+$ and $n^-$; hence these are all regular semisimple elements of $G$. Then, by Gow's Theorem, we can find elements $x_1, y_1 \in G$, both of order $\frac{q+1}{2}$, whose product has order $n^+$, and elements $x_2,y_2 \in G$, both of order $\frac{q-1}{2}$, whose product has order $n^-$.

The only maximal subgroups of $G$ containing elements of order $n^+$ are normalisers of the cyclic subgroup which they generate of order $6n^+$. Similarly for elements of order $n^-$, they are contained in maximal subgroups of order $6n^-$. Since $n^+ - (q+1) = \sqrt{3q}$ we have gcd$(\frac{q+1}{2},n^+)=1$ since neither is divisible by 3. Then, for $q>3$ we have $\frac{q+1}{2} > 6$ so $(x_1,y_1,x_1y_1)$ is indeed an even triple for $G$ of type $(\frac{q+1}{2},\frac{q+1}{2},q+\sqrt{3q}+1)$. Similarly, for $q>3$ we have $\frac{q-1}{2}>6$ and $n^+n^- + (q-1) = q^2$, hence gcd$(\frac{q-1}{2},n^-)=1$ since again neither is divisible by 3. Note this also implies that gcd$(n^+,\frac{q-1}{2})=1$. This gives us an odd triple for $G$ of type $(\frac{q-1}{2},\frac{q-1}{2},q-\sqrt{3q}+1)$.

It is clear that gcd$(\frac{q+1}{2},\frac{q-1}{2})=1$ since their difference is $q$ and neither is divisible by 3. Similarly, gcd$(n^+,n^-)=1$ since their difference is $2\sqrt{3q}$ and both are clearly coprime to 6. Since we have already shown that gcd$(n^+,\frac{q-1}{2})=1$ it remains to show that gcd$(n^-,\frac{q+1}{2})=1$. We have $(q+1) - n^- = \sqrt{3q}$ and since neither is divisible by 3 we are done.\end{proof}

\begin{lemma} The groups $^2G_2(q) \times {}^2G_2(q)$ for $q>3$ admit a mixable Beauville structure of type $$\left(\frac{q+1}{2},\frac{q+1}{2}n^+,\frac{q+1}{2}n^+;\frac{q-1}{2},\frac{q-1}{2}n^-,\frac{q-1}{2}n^-\right)$$ where $n^+ =q+\sqrt{3q}+1$ and $n^- =q-\sqrt{3q}+1$.\end{lemma}
\begin{proof} By Lemmas \ref{cop} and \ref{ree} we need only show that gcd$(\frac{q+1}{2},n^+)=1$ and gcd$(\frac{q-1}{2},n^-)=1$, both of which were demonstrated in the proof of Lemma \ref{ree}. \end{proof}

%%%%%%%%%%%%%%%%%%%%%%%%%%----Large Ree
\subsection{The large Ree groups $^2F_4(2^{2n+1})$}

The large Ree groups $^2F_4(q)$ are defined over fields of order $q = 2^{2n+1}$ for $n \geq 0$ and have order $q^{12}(q^6+1)(q^4-1)(q^3+1)(q-1)$. They are simple except for the case $q=2$ which has simple derived subgroup $^2F_4(2)'$, known as the Tits group, which we consider along with the sporadic groups in the next section. The maximal subgroups of the large Ree groups can be found in \cite{2f4} or \cite{raw}.

\begin{lemma} \label{Ree} The large Ree groups $^2F_4(q)$ for $q >2$ admit a mixable Beauville structure of type $$\left(10,10,n^+;\frac{q^2-1}{3}, n^-,n^-\right)$$ where $n^+ = q^2+q+1+\sqrt{2q}(q+1)$ and $n^- =q^2+q+1-\sqrt{2q}(q+1)$. \end{lemma}
\begin{proof} Let $G = {}^2F_4(q)$ for $q > 2$. Elements of order $10$ exist since $G$ contains maximal subgroups of the form $^2B_2(q) \wr 2$, as do elements of order $\frac{q^2-1}{3}$ since $G$ contains maximal subgroups of the form $SU_3(q):2$ and $PGU_3(q):2$ and since gcd$(3,q+1) = 3$. The only maximal subgroup containing an element of order $n^+$ is the normaliser of the cyclic subgroup which it generates, of order $12n^+$. Similarly, elements of order $n^-$ are only contained in the normalisers of the cyclic subgroup they generate, of order $12n^-$.

Using the computer program {\sf CHEVIE} it is possible to determine the structure constant for a pair of elements of order 10 whose product is $n^+$ and we see this is nonzero. Since $n^+n^- = q^4-q^2+1$ and $q \equiv \pm 2$ mod $5$ we have that gcd$(10,n^-)=$ gcd$(10,n^+)=1$. Then, since no maximal subgroup contains elements of orders $10$ and $n^+$ this is indeed an even triple for $G$. Elements of order $\frac{q^2-1}{3}$ are semisimple and the order of the centraliser of an element of order $n^-$ is $n^-$ \cite{sh} so these are regular semisimple elements. Then by Gow's Theorem there exists a pair of elements of order $n^-$ whose product has order $\frac{q^2-1}{3}$. To show that such a triple will generate $G$ we show that no maximal subgroup contains elements of orders $n^-$ and $\frac{q^2-1}{3}$. Since $(n^+n^-) + (q^2 -1) = q^4$ any common factor of $n^-$ and $\frac{q^2-1}{3}$ must be a power of 2, but since $n^-$ and $q^2-1$ are both odd we have gcd$(\frac{q^2-1}{3},n^-)=1$. Note that this also implies gcd$(n^+,\frac{q^2-1}{3})=1$. Then, since $\frac{q^2-1}{3} > 12$ for $q>2$ we then have that an odd triple of type $(\frac{q^2-1}{3}, n^-,n^-)$ exists for $G$.

We have already shown that gcd$(10,n^-)=1$, gcd$(n^+,q^2-1)=1$ and it is clear that gcd$(10,\frac{q^2-1}{3})=1$. Finally, let $c =$ gcd$(n^+,n^-)$ and note that $c$ is odd. Since $n^+ - n^- = 2 \sqrt{2q}(q+1)$, $c$ must divide $q+1$. Also, since $n^+ + n^- = 2(q^2 + q +1)$, $c$ must also divide $q^2 + q + 1$. Therefore $c$ must divide $q^2$ and hence $c=1$ so we have our desired mixable Beauville structure. \end{proof}

\begin{lemma} The groups $^2F_4(q) \times {}^2F_4(q)$ for $q>2$ admit a mixable Beauville structure of type $$\left(10,10n^+,10n^+;\frac{q^2-1}{3}n^-,\frac{q^2-1}{3}, n^-,n^-\right)$$ where $n^+ = q^2+q+1+\sqrt{2q}(q+1)$ and $n^- =q^2+q+1-\sqrt{2q}(q+1)$. \end{lemma}
\begin{proof} By Lemmas \ref{cop} and \ref{Ree} we need only show that gcd$(10,n^+)=1$ and gcd$(\frac{q^2-1}{3},n^-)=1$, both of which can be found in the proof of Lemma \ref{Ree}. \end{proof}

%%%%%%%%%%%%%%%%%%%%%%%%%%----Steinberg Triality
\subsection{The Steinberg triality groups of type $^3D_4(q)$}

The Steinberg triality groups $^3D_4(q)$ are defined over fields of order $q$ and have order $q^{12}(q^8+q^4+1)(q^6-1)(q^2-1)$. They are simple for all prime powers, $q$, and their maximal subgroups can be found in \cite{3d4} or \cite{raw}.

\begin{lemma} Let $G$ be the Steinberg trialty group $^3D_4(2)$. Then both $G$ and $G \times G$  admit a mixable Beauville structure.\end{lemma}
\begin{proof} It can be verified using GAP that $(a,(ab)^3b^2;ab,b^{ab^2})$, where $a$ and $b$ are the standard generators as found in \cite{brauer}, is a mixable Beauville structure of type $(2,7,28;13,9,13)$ and by Lemma \ref{cop} this yields a mixable Beauville strucure of type $(14,14,28;117,117,13)$ on $G \times G$. \end{proof}

\begin{lemma} \label{tri} Let $G$ be the Steinberg triality group of type $^3D_4(q)$ for $q \geq 3$ and $d=$ gcd$(3,q+1)$. Then \begin{enumerate}
\item for $p=2$ there exists a mixable Beauville structure on $G$ of type $$(6,6,\Phi_{12}(q);\Phi_3(q),\Phi_3(q),\Phi_6(q)),$$
\item for $p \neq 2$ there exists a mixable Beauville structure on $G$ of type $$\left(\frac{q^2-1}{d},\frac{q^2-1}{d},\Phi_{12}(q);\Phi_3(q),\Phi_3(q),\Phi_6(q)\right).$$
\end{enumerate}
\end{lemma}
\begin{proof} Let $G$ be as in the hypothesis. By \cite[Lemma 5.24]{fmp} for $q > 2$ there exists a hyperbolic generating triple of type $(\Phi_3(q),\Phi_3(q),\Phi_6(q))$.

For $p = 2$ one can verify using {\sf CHEVIE} to compute the structure constants that there exist pairs of elements of order $6$ whose product has order $\Phi_{12}(q)$ and it is clear from the list of maximal subgroups that this is indeed an even triple for $G$.

For $p \neq 2$ elements of order $\frac{q^2-1}{d}$ exist since $G$ contains subgroups isomorphic to $SU(3,q)$. Using {\sf CHEVIE} and the list of maximal subgroups it can be shown that an odd triple of type $(\frac{q^2-1}{d},\frac{q^2-1}{d},\Phi_{12}(q))$ exists for $G$. By Zsigmondy's theorem, Theorem \ref{zsig}, it is also clear that we have coprimeness for both Beauville structures. \end{proof}

\begin{lemma} Let $G$ be the Steinberg triality groups of type $^3D_4(q)$ for $q \geq 3$ and $d=$ gcd$(3,q+1)$. Then \begin{enumerate}
\item for $p=2$ there exists a mixable Beauville structure on $G \times G$ of type $$(6,6\Phi_{12}(q),6\Phi_{12}(q);\Phi_3(q),\Phi_3(q)\Phi_6(q),\Phi_3(q)\Phi_6(q));$$
\item for $p \neq 2$ there exists a mixable Beauville structure on $G$ of type $$\left(\frac{q^2-1}{d},\Phi_{12}(q)\frac{q^2-1}{d},\Phi_{12}(q)\frac{q^2-1}{d};\Phi_3(q),\Phi_3(q)\Phi_6(q),\Phi_3(q)\Phi_6(q)\right).$$
\end{enumerate}
\end{lemma}
\begin{proof} By Lemmas \ref{cop} and \ref{tri} we need only verify that gcd$(6,\Phi_{12}(q))=1$ for $p=2$ as the rest follows by construction. This is clear since $\Phi_{12}(q)$ is both odd and coprime to $q^2-1$ which is divisible by 3. \end{proof}

\section{The Sporadic groups}%%%%%%%%%%%%%%%%%%%%%%%%%%----SPORADICS----%%%%%%%%%%%%%%%%%%%%%%%%%%

\begin{table}
\centering
\begin{tabular}{l c c c c c} \hline
$G$			&$x_1$						&$y_1$								&$x_2$						&$y_2$\\ \hline
$M_{11}$		&$ab(ab^2)^2$				&$(ab(ab^2)^2)^{b}$				&$(ab)^5$					&$[a,b]^2$\\
$M_{12}$		&$(ab)^4ba(bab)^2b$		&$x_1^{b^2aba}$					&$ab$						&$ba$\\
$J_1$			&$(ab)^2(ba)^3b^2ab^2$	&$b^{aba}$							&$ab$						&$[a,b]$\\
$M_{22}$		&$(ab)^3b^2ab^2$			&$ba(bab^2)^2a$					&$ab$						&$(ab)^4b^2ab^2$\\
$J_2$			&$ab(abab^2)^2ab^2$		&$x_1^{ab^2}, x_1^{(ba)^2b^2}$	&$ab$						&$ba$\\
$M_{23}$		&$abab^2$					&$babab$							&$ab$						&$(abab)^{ba}$\\
$^2F_4(2)'$	&$(ab)^3bab$				&$(x_1^7)^{baba}$					&$ab$						&$ba$\\
$HS$			&$abab^3$					&$b^3aba$							&$(ab)^3b$					&$((ab)^3b)^b$\\
$J_3$			&$(ab)^2(ba)^3b^2$		&$x_1^b$							&$ab$						&$(ab)^b$\\
$M_{24}$		&$(ab)^4b$					&$(x_1^3)^b$						&$ab$						&$ba$\\
$McL$			&$ab^2$					&$bab$								&$ab$						&$(ab)^b$\\
$He$			&$ab^3$					&$ab^4$							&$ab$						&$ba$\\
$Ru$			&$b$						&$b^{(ab)^5}$						&$(ab)^2$					&$(ba)^2$\\
$Suz$			&$(ab(abab^2)^2)$			&$x_1^{abab^2}$					&$ab$						&$ba$\\
$O'N$			&$[a,b]$					&$([a,b]^2)^{bab}$					&$ab^2$					&$(ab^2)^{abab}$\\
$Co_3$			&$ab$						&$(ab)^{b^2}$						&$a(ab)^2b(a^2b)^2b$		&$(x_2^3)^{baba}$\\
$Co_2$			&$a$						&$b$								&$ab(ab^2)^2b$			&$(x_2^2)^{bab^2}$\\
$Fi_{22}$		&$(ab)^3b^3$				&$((ab)^3b^3)^a$					&$b$						&$b^a$\\
$HN$			&$a$						&$b$								&$[a,b]^a$					&$(ab)^2b((ab)^5(ab^2)^2)^2$\\
$Ly$			&$a$						&$b$								&$abab^3$					&$x_2^{(ab)^7}$\\
$Th$			&$[a,b]$					&$[a,b]^{(ba)^4b^2}$				&$ababa$					&$(x_2^5)^{bab}$\\
$Fi_{23}$		&$a$						&$b$								&$((ab)^{11}b)^3$			&$x_2^a$\\
$Co_1$			&$a$						&$b$								&$[(ab)^3,aba]$			&$[(ab)^{23},ab^2]^{babab}$\\
$J_4$			&$a$						&$b$								&$a(bab)^3(ab)^2b$		&$x_2^a$\\
$Fi_{24}'$		&$ab$						&$((ab)^6b)^{15} $					&$b(ba)^3$					&$(x_2^3)^{bab}$\\ \hline
\end{tabular}
\caption{Words in the standard generators of $G$ \cite{brauer} where $(x_1,y_1,x_2,y_2)$ is a mixable Beauville structure for $G$ of the specified type.} \label{table}
\end{table}

\begin{table}
\centering
\begin{tabular}{l c c c c c} \hline
$G$			&Type of $G$				&Type of $G \times G$\\ \hline
$M_{11}$		&$(8,8,5;11,3,11)$			&$(8,40,40;11,33,33)$\\
$M_{12}$		&$(8,8,5;11,11,3)$			&$(8,40,40;11,33,33)$\\
$J_1$			&$(10,3,10;7,19,19)$		&$(10,30,30;133,133,19)$\\
$M_{22}$		&$(8,8,5;11,7,7)$			&$(8,40,40;77,77,7)$\\
$J_2$			&$(10,10,10;7,7,3)$		&$(10,10,80;7,21,21)$\\
$M_{23}$		&$(8,8,11;23,23,7)$		&$(8,88,88;23,161,161)$\\
$^2F_4(2)'$	&$(8,8,5;13,13,3)$			&$(8,40,40;13,39,39)$\\
$HS$			&$(8,8,15;7,7,11)$			&$(8,120,120;7,77,77)$\\
$J_3$			&$(8,8,5;19,19,3)$			&$(8,40,40;19,57,57)$\\
$M_{24}$		&$(8,8,5;23,23,3)$			&$(8,40,40;23,69,69)$\\
$McL$			&$(12,12,7;11,11,5)$		&$(84,84,12;55,55,11)$\\
$He$			&$(8,8,5;17,17,7)$			&$(8,40,40;17,119,119)$\\
$Ru$			&$(4,4,29;13,13,7)$		&$(4,116,116;13,91,91)$\\
$Suz$			&$(8,8,7;13,13,3)$			&$(8,56,56;13,39,39)$\\
$O'N$			&$(12,6,31;19,19,11)$		&$(12,186,186;209,209,19)$\\
$Co_3$			&$(14,14,5;23,23,9)$		&$(14,70,70;23,207,207)$\\
$Co_2$			&$(2,5,28;23,23,9)$		&$(10,10,28;23,207,207)$\\
$Fi_{22}$		&$(16,16,9;13,13,11)$		&$(144,144,16;143,143,13)$\\
$HN$			&$(2,3,22;5,19,19)$		&$(6,6,22;95,95,19)$\\
$Ly$			&$(2,5,14;67,67,37)$		&$(10,10,14;2479,2479,67)$\\
$Th$			&$(10,10,13;19,19,31)$		&$(130,130,10;589,589,19)$\\
$Fi_{23}$		&$(2,3,28;13,13,23)$		&$(6,6,28;299,299,13)$\\
$Co_1$			&$(2,3,40;11,13,23)$		&$(6,6,40;143,143,23)$\\
$J_4$			&$(2,4,37;43,43,23)$		&$(4,74,74;989,989,43)$\\
$Fi_{24}'$		&$(29,4,20;33,33,23)$		&$(116,116,20;759,759,33)$\\
$\BB$			&$(2,3,8;47,47,55)$ 		&$(6,6,8;47,2585,2585)$\\
$\MM$			&$(94,94,71;21,39,55)$		&$(94,6674,6674;21,2145,2145)$\\ \hline
\end{tabular}
\caption{Types of the mixable Beauville structures for $G$ and $G \times G$ from the words in Table \ref{table} and Lemmas \ref{baby} and \ref{monster}.}
\end{table}

In this section we exhibit explicit words in the standard generators \cite{brauer} of mixable Beauville structures for the sporadic groups. With the exception of the even triple for Janko's group, $J_2$, at least two elements of every triple have coprime order which, by Lemma \ref{cop}, automatically gives us a corresponding triple for $G \times G$. In the case of $J_2$ we have two even triples, $(x_1,x_1^{ab^2})$ and $(x_1,x_1^{(ab)^2b^2})$ of types $(10,10,10)$ and $(10,10,8)$, which are inequivalent by Lemma \ref{inequi}. For groups of reasonable order it can be explicitly computed in GAP that these elements generate. For groups of larger order we appeal to the list of maximal subgroups as found in \cite{ATLAS} and \cite{raw} to ensure generation. For the Monster $\MM$ and the Baby Monster $\BB$ we use less constructive methods to prove the existence of these structures.

\begin{lemma} \label{baby} There exists a mixable Beauville structure on the Baby Monster, $\BB$, and on $\BB \times \BB$. \end{lemma}
\begin{proof} From \cite{sgbm} we know there exists a hyperbolic generating triple of type $(2,3,8)$. Let $$x=(ab)^3(ba)^4b(ba)^2b^2, \; \; y=x^{ab^2}$$ be words in the standard generators \cite{brauer}. They both have order $47$ and their product has order 55, then from the list of maximal subgroups \cite{raw} they will generate $\BB$. This gives a mixable Beauville structure of type $(2,3,8;47,47,55)$ on $\BB$ and by Lemma \ref{cop} we have a mixable Beauville structure of type $(6,6,8;47,2585,2585)$ on $\BB \times \BB$.\end{proof}

\begin{lemma} \label{monster} There exists a mixable Beauville structure on the Monster, $\MM$, and on $\MM \times \MM$. \end{lemma}
\begin{proof} Norton and Wilson \cite[Theorem 21]{norwil} show that the only maximal subgroups of the Monster which contain elements of order $94$ are copies of $2 \cdot \BB$ which does not contain elements of order $71$. Then from a computation of the structure constants an even triple of type $(94,94,71)$ can be shown to exist. Finally, in \cite{exc} it is shown that there exists a hyperbolic generating triple of type $(21,39,55)$ on $\MM$. Therefore we have a mixable Beauville structure of type $(94,94,71;21,39,55)$ on $\MM$ and of type $(94,6674,6674;21,2145,2145)$. This completes the proof. \end{proof}

Finally we have the following Lemma which completes the proof of Theorem \ref{main}.

\begin{lemma} Let $G$ be one of the 26 sporadic groups or the Tits group $^2F_4(2)'$. Then there exists a mixable Beauville structure on $G$ and $G \times G$. \end{lemma}

\end{document}